\documentclass[letterpaper, 10 pt, conference]{ieeeconf}  
\usepackage{amsmath}
\usepackage{amssymb}
\IEEEoverridecommandlockouts                              
\usepackage{graphics} 
\usepackage{epsfig} 
\usepackage{epstopdf}
\usepackage{mathptmx} 
\usepackage{times} 
\usepackage{amsmath} 
\usepackage{amssymb}  
\usepackage[linesnumbered,lined,boxed,commentsnumbered]{algorithm2e}
\usepackage[noend]{algpseudocode}
\usepackage{subfigure}

\newtheorem{mydef}{Definition}
\newtheorem{mythm}{Theorem}

\newtheorem{mylem}{Lemma}

\newtheorem{mycor}{Corollary}

\newcommand{\mbf}[1]{\mathbf{#1}}

\let\bbl\Bigl

\let\bbr\Bigr

\newcommand{\R}{\mathbb{R}}
\newcommand{\x}{\mathbf{x}}
\newcommand{\one}{\mathbf{1}}

\newcommand{\N}{\mathbb{N}}

\newcommand{\norm}[1]{\lVert{#1}\rVert}

\title{\LARGE \bf
Combining SOS and Moment Relaxations with Branch and Bound to Extract Solutions to Global Polynomial Optimization Problems
}

\author{Hesameddin Mohammadi, Matthew M. Peet
}

\begin{document}

\maketitle
\thispagestyle{empty}
\pagestyle{empty}

\begin{abstract}

In this paper, we present a branch and bound algorithm for extracting approximate solutions to Global Polynomial Optimization (GPO) problems with bounded feasible sets.
 The algorithm  is based on a combination of SOS/Moment relaxations and successively bisecting a hyper-rectangle containing the feasible set of the GPO problem.  At each iteration, the algorithm makes a comparison between the volume of the hyper-rectangles and their associated lower bounds to the GPO problem obtained by SOS/Moment relaxations to choose and subdivide an existing hyper-rectangle. For any desired accuracy, if we use sufficiently large order of SOS/Moment relaxations, then the algorithm is guaranteed to return a suboptimal point in a certain sense.  For a fixed order of SOS/Moment relaxations, the complexity of the algorithm is  linear in the number of iterations and polynomial in the number of constrains.  We illustrate the effectiveness of the algorithm for a 6-variable, 5-constraint GPO problem for which the ideal generated by the equality constraints is not zero-dimensional - a case where the existing Moment-based approach for extracting the global minimizer might fail.

\end{abstract}

\section{Introduction }
\textit{Global Polynomial Optimization} (GPO) is defined as optimization of the form
\vspace{-0 mm}
\begin{align}
\label{eq:GPO}
f^*:= \underset{x\in\mathbb{R}^n}{\min.} & \quad  f(x)\\
\text{subject to } & \;\;\text{ $ g_i(x)\ge 0$}\quad\, \text{ for } i=1,\cdots,s \nonumber\\
& \;\; \;\text{$ h_j(x) = 0$}\quad \text{ for } j=1,\cdots,t\nonumber \text{,}
\end{align}
where $f$, $g_i$, and $h_i$ are real-valued polynomials in decision variables $x$.  
As defined in Eq.~\ref{eq:GPO}, the GPO problem encompasses many well-studied sub-classes including  Linear Programming (LP), Quadratic Programming (QP), Integer Programming (IP), Semidefinite Programming (SDP) and Mixed-Integer Nonlinear Programming (MINLP)~\cite{Laurent2009}. Because of its generalized form, almost any optimization problem can be cast or approximately cast as a GPO, including certain NP-hard problems from economic dispatch~\cite{dispatch}, optimal power flow~\cite{powerflow} and optimal decentralized control~\cite{lavaei2013optimal}. As applied to control theory, GPO can be used for stability analysis of polynomial dynamical systems by, e.g., verifying polytopic invariants as in~\cite{BenSassi20123114}.

Although GPO is NP-hard, there exist a number of heuristics and algorithms that can efficiently solve special cases of GPO. For example, LP~\cite{LinearProgram,LinearProgram2}, QP~\cite{quadraticGPO}, and SDP all have associated polynomial-time algorithms. More broadly, if the feasible set is convex and the objective function is convex, then barrier functions and descent methods will typically yield a computationally tractable algorithm. When the problem is not convex, there also exist special cases in which the GPO problem is solvable. For example, in~\cite{groebnerbasis} the unconstrained problem was solved by parameterizing the critical points of the objective function via Groebner bases. In the special case of $ x\in\R^1 $, the problem was solved in~\cite{UnivariateCase,shor1987quadratic}. In addition, there exist several widely used heuristics which often yield reasonably suboptimal and approximately or exactly feasible solutions to the GPO problem (e.g.~\cite{SCIP1,SCIP2}), but which we will not discuss here in depth.

If we expand our definition of algorithm to include those with combinatorial complexity but finite termination time, then if the feasible set of the GPO problem is compact and the ideal generated by the equality constraints is radical and zero dimensional (typical for integer programming~\cite{grid}), then one may use the moment approach to solving sequential Greatest Lower Bound (GLB) problems to obtain an algorithm with finite termination time. Unfortunately, however, it has been shown that the class of problems for which these methods terminate is a strict subset of the general class of GPO problems~\cite{Nie2013} and furthermore, there are no tractable conditions verifying if the algorithm will terminate or bounds on computational complexity in the case of finite termination. In this paper, however, we take the idea of generating Greatest Lower Bounds and propose an alternative method for extracting approximate solutions that does not require finite convergence and, hence, has polynomial-time complexity.


Let
\vspace{-4.5 mm}
\begin{equation}
\label{feasibleSet}
S:=\{x \in \R^n\,:\, g_i(x)\ge 0,\; h_j(x)=0\}
\end{equation}
be the feasible set and $ x^* $ be an optimal solution of GPO Problem~\eqref{eq:GPO}. The Greatest Lower Bound (GLB) problem associated to GPO Problem~\eqref{eq:GPO} is defined as
\begin{align}
\label{eq:GLB}
\lambda^*:=&\underset{\lambda\in\mathbb{R}}{\max.} \quad  \lambda\\
&\text{subject to }\quad  f(x)-\lambda > 0 \quad, \forall x\in S. \nonumber
\end{align}
The GLB and GPO problems are closely related, but are not equivalent. For example, it is clear that $\lambda^* = f^*=f(x^*) $, where $f^*$ is as defined in Eq.~\eqref{eq:GPO}. Furthermore, as discussed in Section~\ref{section:BranchBound}, an algorithm which solves the GLB problem with complexity $O(k)$ can be combined with Branch and Bound to approximately solve the GPO for any desired level of accuracy $\epsilon$ with complexity  $O(\log(1/\epsilon)k)$, where we define an approximate solution to the GPO as a point $x \in \R^{n}$ such that $|x-x^*|\le \epsilon$ and $|f(x)-f(x^*)|<\epsilon$.


Many convex approaches have been applied to solving the GLB problem, all of which are based on parameterizing the cone of polynomials which are positive over the feasible set of the corresponding GPO problem. The most well-known of these approaches are Sum of Squares (SOS) programming~\cite{parrilo2000structured}, and its dual Moment-relaxation problem~\cite{lasserre2009moments}. Both these approaches are well-studied and have implementations as Matlab toolboxes, including SOSTOOLS~\cite{sostools} and Gloptipoly~\cite{Gloptipoly}. These approaches both yield a hierarchy of primal/dual semidefinite programs with an increasing associated sequence of optimal values $\{p^*_k\}_{k\in\mathbb{N}}$ (SOS) and $\{d^*_k\}_{k\in\mathbb{N}}$ (Moment) such that if we denote the OBV of the corresponding GPO problem  by $ f^* $, then under mild conditions, both sequences satisfy $p^*_k\le d_k^*\le f^*$ and $\lim\limits_{k\rightarrow\infty}p^*_k=\lim\limits_{k\rightarrow\infty}d^*_k = f^*$~\cite{schweighofer2005optimization}.
Moreover, if the feasible set of GPO Problem~\eqref{eq:GPO}, $ S $, as defined in~\eqref{feasibleSet},  is nonempty and compact, then there exist bounds on the error of SOS/Moment relaxations which scale as $|p_k^*-f^*|\cong \frac{c_2}{\sqrt[c_1]{\log(k)}}$\vspace{0.02in} for constants $c_1$ and $ c_2 $ that are functions of polynomials $ f $, $g_i$ and $h_j$~\cite{complexity}. The goal of this paper, then, is to combine the SOS and Moment approaches with a branch and bound methodology to create an algorithm for solving the GPO problem and extracting a solution.


Specifically, in Sec.~\ref{sec:ModBranchBound},  we propose a sequence of branch and bound algorithms, denoted by $ E_k $, such that for any $ k\in\N $ and for a given  GPO problem of  Form~\eqref{eq:GPO} with a bounded feasible set, Algorithm $ E_k $ in polynomial-time returns a point $ x_k\in\R^n $ that is sub-optimal to the GPO problem in the following sense. If $x_k$ is the sequence of proposed solutions produced by the sequence of algorithms $E_k$, we show that if the feasible set, $ S $, of Problem~\eqref{eq:GPO} is bounded, then there exist a sequence of feasible point $ y_k\in S$  such that  $\underset{k\rightarrow\infty}{lim}||x_k-y_k|| = 0$ and $ \underset{k\rightarrow\infty}{lim}f(y_k) = \underset{k\rightarrow\infty}{lim}f(x_k)= f^*$, where $ f^* $ is the OBV of the GPO problem.

The sequence of algorithms can be briefly summarized as follows. Algorithm $ E_k $ initializes with hyper-rectangle  $C_1:=\{ x\,|\,(x_i-\underbar c_{i,1})(\bar c_{i,1}-x_i)\ge 0, \, i=1,\cdots,n\}$ such that $ S\subset C_1 $. Then, at iteration $ m $, the algorithm forms two new hyper-rectangles by bisecting $ C_m $ by its longest edge ($\bar c_i-\underbar c_i$) and intersecting each of the new hyper-rectangles with $S$ to create two new GPO problems. Then, the algorithm computes a GLB estimate of the OBV of the new GPO problem by solving the corresponding $ k $'th-order SOS/Moment relaxations. If $ \lambda^*_m $ denotes the best lower bound to $ f^* $, obtained up to iteration $ m $, then $ C_{m+1} $ is determined to be the existing hyper-rectangle with a smallest volume subject to the constraint that the corresponding GLB is less that $\frac{m\,\eta}{l}+\lambda^*_m $, where  $ \eta>0  $ and  $ l\in\N $ are the design parameters.  Finally, after certain number of iterations, the algorithm terminates by returning the centroid of the last hyper-rectangle.

In Sec.~\ref{section: Complexity} we discuss the complexity of the proposed algorithm by first showing that for any $ k\ge 2 $, the feasible set of the $ k $'th order SOS relaxation associated to a branch is completely contained in that of each of its subdivided branches - implying that the lower bounds obtained by both the SOS and Moment relaxation (due to the duality between SOS and Moment methods) are increasing. Of course, at those branches with $ S\cap C_m=\emptyset $, these bounds will approach $ +\infty $. Next, we show that if the error of  the $ k $'th order SOS/Moment relaxations associated to the hyper-rectangles obtained through Algorithm $ E_k $ is bounded by $\frac{c_2}{\sqrt[c_1]{\log(k)}}\le\eta/(l+1)$, then each of the bisected hyper-rectangles is guaranteed to contain a feasible point that is $\eta$-suboptimal. Therefore, the feasible set will be reduced in volume as $\frac{1}{2^{m}}$, up to iteration $ m $. In other words, the number of iterations necessary to achieve a hyper-rectangle with a longest edge of the length $\epsilon$ is logarithmic in $1/\epsilon$.  Finally, since at each iteration the number of constraints is fixed and the complexity is linear in the number of iterations, we conclude that algorithm $E_k$ is polynomial-time.

In Sec.~\ref{section: NE}, we illustrate the effectiveness of the proposed algorithm by applying it to an example problem wherein the existing  Moment based approach fails to extract a solution. Finally we conclude in Sec.~\ref{section:Conclusion}.

\section{Notation}
\label{sec:Notation}
Let $\mathbb{N}^n$ be the set of n-tuples of natural numbers. We use $\mathbb{S}^n$ and $\mathbb{S}^{n+}$ to denote the symmetric matrices and cone of positive semidefinite matrices of size $n\times n$, respectively.
For any $ a,\,b\in \R^n $, we denote by $ C(a,b) $ the hyper-rectangle $ \{x\in\R^n|a\le x \le b\} $, where $y\ge 0$ is defined by the positive orthant. We use $\ell_{\infty}$ to denote the set of bounded infinite sequences. For any $ k\,,n\in\N $, let $\mathbb{N}^n_{(k)}:=\{b\in \mathbb{N}^n: |b|_{\ell_1}\leq k\}$, where $|b|_{\ell_1}:=\sum_{i=1}^n |b_i|$. Finally, we  denote the ring of multivariate polynomials with real coefficients as $\mathbb{R}$[$x$].

\section{Problem Statement}
\label{sec:ProblemStatement}
In this paper, we consider simplified GPO problems of the form:
\vspace{-3 mm}
\begin{align}
\vspace{-0.5 in}
\label{eq:GPOAsli}
f^*:= \underset{x\in\mathbb{R}^n}{\min.} & \quad  f(x)\\
\text{subject to } & \;\;\text{ $ g_i(x)\ge 0$}\quad\, \text{ for } i=0,\cdots,s \nonumber
\end{align}
where $f,g_i\in\mathbb{R}\text{[}x\mathbb{]}$. The class of problems in~\eqref{eq:GPOAsli} is equivalent to that in~\eqref{eq:GPO}, where we have simply replaced every $ h_i(x)=0 $ constraint with some $g_1(x)=h(x)\ge 0$ and $g_2(x)=h(x)\le 0$. For every problem of Form~\eqref{eq:GPOAsli}, we define the associated feasible set $
S:=\{x \in \R^n\,:\, g_i(x)\ge 0\}.$

In this paper, we assume $S\neq \emptyset$. Note that given $g_i$, one may use SOS optimization combined with Positivstellensatz results~\cite{stengle1974nullstellensatz} to determine feasibility of $S$.

\noindent \textbf{Proposed Algorithm} In this paper, we propose a GLB and Branch and Bound-based algorithm which, for any given $\epsilon> 0$, will return some $x\in\mathbb{R}^n$
for which there exists a point $y\in S $ such that:
\vspace{- 2 mm}
\begin{align}
\label{prop:output}
f(y)-f^* \leq \epsilon, \qquad \text{and} \qquad \quad\Vert y-x\Vert<\epsilon.
\end{align}
Furthermore $x$ itself is $\epsilon$-suboptimal in the sense that $\vert f(x)-f^*\vert \leq \epsilon$ and $g_i(x)\ge -\epsilon$.

Before defining this algorithm, however, in the following section, we describe some background on the dual SOS and Moment algorithms for generating approximate solutions of the GLB Problem.

\section{Background on Semidefinite Representations of SOS/Moment Relaxations}
In this section, we describe two well-known asymptotic algorithms which are known to generate sequences of increasingly accurate suboptimal solutions to the GLB problem - namely the SOS  and Moment approaches. Both these methods use Positivstellensatz results which parameterize the set of polynomials which are positive over a given semialgebraic set.

\subsection{Sum-of-Squares Polynomials}
In this subsection, we briefly define and denote sets of sums of squares of polynomials.

We denote monomials in variables $x \in \R^n$ as $x^\alpha:=\prod_{i=1}^n x_i^{\alpha_i}$ where $\alpha \in \N^n$. Monomials can be ordered using various orderings on $\N^n$. In this paper, we use the graded lexicographical ordering. This ordering is defined inductively as follows. For $a,b \in \N^n$,  $a\le b $ if  $\sum_{i=1}^{n}a_i < \sum_{i=1}^{n}b_i,$ or $a_1 = b_1$ and $[a_2,\cdots, a_n]\le [b_2,\cdots, b_n]$. Denote by $Z(x)$ the infinite ordered vector of all monomials, where $x^\alpha < x^\beta $ if $\alpha < \beta$. Because we have used the graded lexicographical ordering, if we restrict ourselves to the first $\binom{d+n}{d}$ elements of $Z$, then this is the vector of all monomials of degree $d$ or less. We denote this truncated vector as $Z_d(x)$ and the length of $Z_d$ as $\Lambda(d):=\binom{d+n}{d}$. Using this definition, it is clear that any polynomial can be represented as $p(x)=c^T Z_d(x)$ for some $c \in \R^{\Lambda(d)}$, where $d$ is the degree of $p$.

The vector of monomials can also be combined with positive matrices to completely parameterize the cone of sums-of-squares polynomials. Formally, we can denote the subset of polynomials which are the sum of squares of polynomials as
\vspace{-2.5   mm}
\begin{align}
\Sigma_S:=\{s \in \R[x]\,:\, s(x) =\sum_{i=1}^l p^2_i(x),\, p_i \in \R[x], \, l \in \N\}.
\end{align}
Clearly any element of $\Sigma_S$ is a nonnegative polynomial. Furthermore, if $p \in \Sigma_S $ and is of degree $2d$, then there exists a positive semidefinite matrix $\Omega\in\mathbb{S}^{\Lambda(d)+}$ such that
\vspace{-0.5   mm}
\[
p(x) = Z_d(x)^T \Omega Z_{d}(x).
\]
Conversely, any polynomial of this form, with $\Omega\ge 0$, is SOS. This parametrization of SOS polynomials using positive matrices will allow us to convert the GLB problem to an LMI. However, before defining this LMI approach, we must examine the question of positivity on semialgebraic subsets of $\R^n$.


%
%

\subsection{Putinar's Positivstellensatz, Quadratic Modules and the Archimedean Property}
Sum-of-Squares polynomials are globally non-negative. In this section, we briefly review Putinar's positivstellensatz which gives necessary conditions for a polynomial to be positive on the semiaglebraic set $  S:=\{x \in \mathbb{R}^n: g_i(x)\ge 0,\; i=1,\dots,s\},$ where $S\neq \emptyset$ and is compact.

Putinar's Positivstellensatz uses the $g_i$ which define $S$ to deduce a cone of polynomials which are non-negative on $S$. This cone is the quadratic module which we define as follows.
\begin{mydef}
	\label{def:kQuadM}
	Given a finite collection of polynomials $g_i\in \mathbb{R}[x]$, we define the quadratic module as
	\vspace{-2.5 mm}
	\begin{align*}
	M:= & \{p|
	p=\sigma_0+\sum_{i=1}^{s}\sigma_i g_i \quad\sigma_i \in \Sigma_S
	\},
	\end{align*}
	\vspace{-2 mm}
	and the degree-$k$ bounded quadratic module as
	\begin{align*}
	M^{(k)}:= & \{p|
	p=\sigma_0+\sum_{i=1}^{s}\sigma_i g_i \quad\sigma_i \in \Sigma_S \quad \text{deg}(\sigma_i g_i)\leq k
	\}.
	\end{align*}
\end{mydef}
Clearly, any polynomial in $M$ is non-negative on $S$. Furthermore, since $\Sigma_S$ parameterizes $M$ and positive matrices parameterize $\Sigma_S$, the constraint $p\in M_k$ can be represented as an LMI. Furthermore, if the module satisfies the Archimedean property, then Putinar's Positivstellensatz states that any polynomial which is positive on $S$ is an element of $M$. That is, $M$ parameterizes the cone of polynomials positive on $S$.

A quadratic module $M$ is said to be Archimedean if there exists some $p\in M$ and $R \neq 0$ such that $ p(x)= R^2-\sum_{i=1}^{n}x_i^2 $. We say that $\{g_i\}$ is an Archimedean representation of $S$ if the associated quadratic module is Archimedean. Note that the Archimedean property is a property of the functions $g_i$ which then define the quadratic module and not a property of $S$. Specifically, if $S$ is compact, then there always exists an Archimedean representation of $S$. Specifically, in this case, there exists an $R>0$ such that $\norm{x}\le R$ for all $x \in S$. Now define $ g_{s+1}(x)=R^2-\sum_{i=1}^{n}x_i^2 $.

%

\subsection{SOS approach to solving the GLB problem }
In this subsection, we briefly describe the use of SOS programming to define a hierarchy of GLB problems.

Consider  the GPO Problem~\eqref{eq:GPOAsli} where $\{g_i\}$ is an Archimedean representation of the feasible set, $ S $, with associated quadratic module $M$. We now define the degree-unbounded version of the SOS GLB problem.
\vspace{0   mm}
\begin{align}
\label{eq:GLBinM}
&f^*=\lambda^*:=\underset{\lambda\in\mathbb{R}}{\max.}  \quad  \lambda\quad\quad\quad\quad\\
&\text{subject to }\quad  f(x)-\lambda \in M. \nonumber\quad\quad\quad\quad
\end{align}
Since $M$ is Archimedian, it follows that $\lambda^*=f^*$ (where $f^*$ is as defined in~\eqref{eq:GPOAsli}). Although Problem~\eqref{eq:GLBinM} is convex, for practical implementation we must restrict the degree of the SOS polynomials which parameterize $M$ - meaning, we must restrict ourselves to optimization on $M^{(k)}$. This defines a new sequence of GLB problems as
\vspace{-0   mm}
\begin{align}
\label{eq: sosDegk}
p^*_k:=\quad&\quad \underset{\lambda\in\mathbb{R}}{\max.}  \quad  \lambda\quad\quad\quad\quad\\
&\text{subject to }\quad  f(x)-\lambda \in M^{(k)}. \nonumber\quad\quad\quad\quad
\end{align}
Clearly, $p_i^*\le p_j^* \le \lambda^*$ for any $i<j$. Additionally, it was shown in~\cite{parrilo2000structured} that $\underset{k\rightarrow\infty}{\lim }\;  p_k^* = p^*$.  Furthermore, it was shown in~\cite{schweighofer2005optimization,complexity} that bounds on the convergence rate of $p_k^* \rightarrow \lambda^*$ exist as a function of $g_i$, $f$ and $k$. Finally, the computational complexity of $p_k$ is equivalent to that  of a semidefinite program with order $(s+1)\Lambda(\lceil\frac{k}{2}\rceil)^2$ scalar variables.

%
%
%
%

\subsection{Moment approach to solving the GLB problem}
In this subsection, we briefly describe the Moment approach to solving the GLB problem.

Let  $\mathbb{K}$ denote the set of Borel subsets of $\mathbb{R}^n$ and let $\mathcal{M}(\mathbb{K}) $ be the set of finite and signed Borel measures on $\mathbb{K}$. The following lemma uses the set of probability measures with support on $ S $ to provide a necessary and sufficient condition for a polynomial to be positive over $ S $.
\begin{mylem}
	\label{lem:moment1}
	Given $S\in\mathbb{K}$ and $f:\mathbb{R}^n\rightarrow\mathbb{R}$ integrable over $S$, a polynomial $f(x)$ is nonnegative on $ S $ if and only if $\displaystyle \int_{S}f(x)\; d\mu\;\geq 0$ for all $\mu\in\mathcal{M}(\mathbb{K})$ such that $\mu(S)=1$ and $\mu(\R^n / S)=0$.
\end{mylem}

Again, consider GPO Problem~\eqref{eq:GPOAsli} with the feasible set $ S $. Now, if we define the moment optimization problem
\vspace{-1.   mm}
\begin{align}
\label{eq:mom1}
\lambda^*:= &\quad \underset{\mu\in\mathcal{M}(\mathbb{K}),\lambda\in\mathbb{R}}{\max.} \quad  \lambda\\
&\text{subject to }\quad \int_{S}(f-\lambda)\; d\mu  \ge 0,\nonumber\\
&\mu(S) =1\qquad { and } \qquad \mu(\R^n / S)=0, \nonumber
\end{align}
then Lemma~\ref{lem:moment1} implies, using a duality argument, that $\lambda^*=f^*$  where $f^*$ is as defined in Problem~\eqref{eq:GPOAsli}~\cite{lasserre2009moments}.

Unfortunately, Problem~\eqref{eq:mom1} requires us to optimize over the space of measures $\mu\in\mathcal{M}(\mathbb{K})$. However, as yet we have no way of parameterizing these measures or imposing the constraints $\mu(S) =1$ and $\mu(\R^n / S)=0$. Fortunately, we find that a measure can be parameterized effectively using the moments generated by the measure. That is, any measure $\mu\in\mathcal{M}(\mathbb{K})$ has an associated ordered vector of moments, indexed by $\alpha \in \N^n$ using monomial $x^\alpha$ as
\vspace{-2  mm}
\[
\psi(\mu)_\alpha:=\int_{S}x^\alpha\; d\mu.
\]
Furthermore, if  $\mu(S)=1$ and $\mu(\R^n / S)=0$, then $\psi(\mu)^T c\ge 0 $ for all $c \in C_g$, where we define $ C_g:=\{c : c^T Z(x) \in M\}\subset \ell_{\infty}$, where $M$ is the module defined by the $g_i$ in $S:=\{x: g_i(x)\ge 0\}$. More significantly, the converse is also true. This means that if we add the constraint $\psi_1=1$ to ensure $\mu(S)=1$, then we may replace the measure variable $\mu$  by the moment variable $\psi$, as described.  However, this requires us to enforce the constraint $\psi^T c\ge 0$ for all $c$ such that $c^T Z(x) \in M$. This constraint, however, can be represented using semidefinite programming~\cite{schweighofer2005optimization,lasserre2009moments}. Finally, if $f(x) = d^T Z(x) $, we can enforce the integral constraint of optimization problem~\eqref{lem:moment1} as
\vspace{-2   mm}
\begin{align*}
\int_{S}f(x)-\lambda\; d\mu=d^T \psi -\lambda \ge 0.
\end{align*}
This allows us to formulate the equivalent GLB problem as
%
%
%
%
%
\vspace{-1   mm}
\begin{align}
\label{eq:mom3}
&d^*=\underset{\mathbf{y}\in\ell_\infty, \lambda\in\mathbb{R}}{\max.} \quad  \lambda\\
&\text{subject to }\quad  \mathbf{y}^T d-\lambda\geq 0,\quad \nonumber\\
&\quad\quad\quad\quad\quad\, \mathbf{y}^T c \geq 0\quad,\forall\; c\in C_g,\nonumber\\
&\quad\quad\quad\quad\quad\, \mathbf{y}_{0}=1\nonumber.
\end{align}
Then $d^*=\lambda^*$ and as for the SOS GLB problem, we define a sequence of truncations of the moment GLB problem
\vspace{-1   mm}
\begin{align}
\label{eq:momOfK}
d_k^*:=\;&\underset{\mathbf{y} \in \R^q, \lambda\in\mathbb{R}}{\max.} \quad  \lambda\\
&\text{subject to }\quad  \mathbf{y}^T d-\lambda\geq 0,\quad \nonumber\\
&\quad\quad\quad\quad\quad\, \mathbf{y}^T c \geq 0\quad,\forall\; c\in C^k_g,\nonumber\\
&\quad\quad\quad\quad\quad\, \mathbf{y}_{0}=1\nonumber,
\end{align}
where $q = \Lambda (k)$ and $C^k_g:=\{c \;:\; c^T Z_k(x) \in M^{(k)}\}$.

It has been shown that $p^*_k\leq d^*_k$  for all $k$ and furthermore $\underset{k\rightarrow\infty}{\lim }\;  p_k^*= \underset{k\rightarrow\infty}{\lim }\;  d_k^* = \lambda^*$. Moreover, if the interior of $S$ is not empty, then $ p_k^*= d_k^*$~\cite{schweighofer2005optimization}. Note that $\mathbf{y}^T c \geq 0\quad,\forall\; c\in C^k_g$ is an SDP constraint and the number of decision variables is $\Lambda (k)$. This implies that the computational complexity of $d_k$ and $p_k$ are similar.

In the following section, we combine the GLB problems defined by the SOS/Moment approach with a Branch and Bound sequence to approximately solve the GPO problem.

\section{Solving the GPO Problem using SOS, Moments and Branch and Bound}
\label{section:BranchBound}
In this section, we show that the following algorithm can be used to solve the GPO problem given a solution to the GLB problem.

\noindent \textbf{The Ideal Branch and Bound Algorithm}

At every iteration, we have a hyper-rectangle $A_i=[a_i,b_i]$;
\begin{enumerate}
	\item Initialize the algorithm;
	\item Bisect $A=[a_i,b_i] = [a',b']\cup[a'',b'']=A_1 \cup A_2$;
	\item Compute the Greatest Lower Bound of
	\vspace{-2.5   mm}
	\begin{align}
	\lambda_i^*:=&\underset{\lambda\in\mathbb{R}}{\max.} \quad  \lambda\\
	&\text{subject to }\quad  f(x)-\lambda > 0 \quad, \forall x\in S \cap A_i; \nonumber
	\end{align}
	\item If $\lambda_1^*>\lambda_2^*$, set $A=A_1$, otherwise $A=A_2$;
	\item Goto 2 ;
\end{enumerate}
At termination, we choose any $x \in A$, which will be accurate within  $|x-x^*|\le r 2^{-k/n}$.

Let us examine these steps in more detail.

\noindent \textbf{Initialize the algorithm} Since the set $S$ is compact, there exists some $r>0$ such that $S \subset B_r(0)$. We may then initialize $A=[-r \one , r\one ]$, where $ \one $ is the vector of all 1's.

\noindent \textbf{Bisect} Bisection of the hypercube occurs along the longest edge. Thus, after $n$ iterations, we are guaranteed a two-fold increase in accuracy. As a result, the largest edge of the hypercube diminishes as $2^{-k/n}$.

\noindent \textbf{Compute the Greatest Lower bound} We assume that our solution to the GLB problem is exact. In this case, we are guaranteed that an optimizing $x$ will always lie in $A_i$.

\subsection{Complexity of the Ideal Branch and Bound Algorithm}

In this subsection, we  show that any exact solution to the GLB can be used to solve the GPO problem with arbitrary accuracy in a logarithmic number of steps using the Ideal Branch and Bound algorithm.

Suppose $G$ is a GLB Problem of the Form~\eqref{eq:GLB} with solution $\lambda^*$. Define the algorithm $H: G\mapsto \lambda^*$ as $\lambda^*= H(G)$. Further suppose $H$  has time-complexity $O(k)$, where $k$ is a measure for the size of $G$. In the Ideal Branch and Bound algorithm, if we use $ H $ to perform Step~(3), it is straightforward to show that for any $ \epsilon>0, $ after $m= 2n(c_1+\log\frac{1}{\epsilon}) $ iterations, if   $A_m=[ a,b]$, then $\max_i |b_i-a_i|\le\epsilon$ and GPO Problem~\eqref{eq:GPOAsli} has a minimizer $ x^*\in C(a,b) $, where $c_1$ depends on the size of $S$ and $ n $ is the number of variables. Now since all the GLB problems defined in Step~(3) of the algorithm are of equal size, $ k $, the complexity of the Ideal Branch and Bound algorithm for a given $ \epsilon $, is $ O(2nk(c_1+\log\frac{1}{\epsilon})) $.

In this section, we considered the ideal case when the GLB can be solved exactly. In the following section we adapt this algorithm to the case when sequential algorithms such as SOS/Moment problems are used to solve the GLB problem. In this case, our approach will also be defined as a sequence of approximation algorithms.

\section{Modified Branch and Bound Algorithm}
\label{sec:ModBranchBound}
In this section, we present a slightly modified branch and bound algorithm that combined with SOS/Moment relaxations, can approximate the solution to the GPO problem to any desired accuracy, in a certain sense.

\noindent \textbf{The Modified Branch and Bound Algorithm}
At every iteration, we have an active hyper-rectangle $A=[a,b]$ and a set of feasible rectangles $Z=\{[a_i,b_i]\}_i$ each with associated GLB $\lambda_i$.
\begin{enumerate}
  \item Initialize the algorithm
  \item Bisect $A=[a,b] = [a',b']\cup[a'',b'']=A_1 \cup A_2$
  \item Compute the Greatest Lower Bound of
   \begin{align}
   \lambda_{i}^*:=&\underset{\lambda\in\mathbb{R}}{\max.} \quad  \lambda\\
   &\text{subject to }\quad  f(x)-\lambda > 0 \quad, \forall x\in S \cap A_i. \nonumber
   \end{align}
  \item If $\lambda_i^* \le \lambda^*+\epsilon$, add $A_i$ to $Z$.
  \item Set $A=Z_i$ where $Z_i$ is the smallest element of $Z$.
  \item Goto 2
\end{enumerate}

At termination, we choose any $x \in A$, which will be accurate within  $|x-x^*|\le r 2^{-k/n}$.

\subsection{Problem Definition and SOS/Moment Subroutine}

Consider GPO Problem~\eqref{eq:GPOAsli} and suppose the corresponding  feasible set
$
S:=\{x \in \mathbb{R}^n: g_i(x)\ge 0\},
$
is nonempty and compact with $ S\subset C(a,b) $, for some $ a,b\in\R^n $ with associated Archimedean quadratic module $M$.

Before defining the main sequential algorithm $E_k$, we will define the $k^th$-order SOS/Moment GLB subroutine, denoted $B_k$, which calculates the GLB in Step~(3) of the Modified Branch and Bound Algorithm.

\noindent\textbf{SOS/Moment Subroutine $\mbf{\lambda^*_k=B_k[a,b]}$}

Given $ A = [a,b]$, define the polynomials $w_{i}(x):=(b_i-x_i)(x_i-a_i)$. These polynomials are then used to define the modified feasible set $S \cap A$ as
\begin{align}
     \label{eq:Semialgebraic_v2}
     S_{ab}:=\{x \in \mathbb{R}^n:& g_i(x)\ge 0,\;\forall i: 1\le i\le s,\\
      \;& w_{j}(x)\ge 0,\;\forall j:\; 1\le j\le n \}\nonumber,
\end{align}
and the corresponding modified degree-$ k $ bounded quadratic module as
\begin{align}
\label{eq:ModifiedQuadraticModuleK}
	M_{a\,b}^{(k)}:= & \bbl\{p \, : \, p=\sum_{i=0}^{s}\sigma_i g_i+\sum_{i=s+1}^{s+n}\sigma_i w_i, \quad\sigma_i \in \Sigma_S, \\
&\qquad \qquad \qquad \text{deg}(\sigma_i g_i)\leq k, \;\; \text{deg}(\sigma_i w_i)\leq k
	\bbr\}\nonumber,
\end{align}
where $ g_0(x)=1 $.
This allows us to formulate and solve the modified $k$-th order SOS GLB problem
\begin{align}
p_k^*:=\quad& \underset{\lambda \in\mathbb{R}}{\max.}  \quad  \lambda\quad\quad\quad\quad\\
&\text{subject to }\quad  f(\x)-\lambda \in M_{a\,b}^{(k)} \nonumber\quad\quad\quad\quad
\end{align}
and the corresponding dual GLB moment problem as described in the preceeding section. The subroutine returns the value $\lambda_k^*=p_k^*$.

\subsection{Formal Definition of the Modified Branch and Bound Algorithm, $E_k$}

We now define a sequence of Algorithms $ E_k$ such that for any $ k\in\N $, $ E_k $ takes GPO Problem~\eqref{eq:GPOAsli} and returns an estimated feasible point $ x^*$.

\noindent\textbf{The Sequence of Algorithms $ E_k$:}

In the following, we use the notation $a \leftarrow b$ to indicate that the algorithm takes value $b$ and assigns it to $a$. That is, $a=b$. In addition parameter $ 0<\eta<1 $ represents error tolerance for trimming branches and in Theorem~\ref{theo: complexity} is set by the desired accuracy as $\eta < \epsilon$. The parameter $l$ represents the number of branch and bound loops and in Theorem~\ref{theo: complexity} is set by the desired accuracy as $l > n \log_2(\frac{L\sqrt{n}}{\eta})$ where $n$ is the number of variables and $L$ is a bound on the radius of the feasible set.


The inputs to the following algorithm $E_k$ are the functions $\{g_i\}$ and $f$, the initial hyper-rectangle such that $S \subset [a,b]$, and the design parameters $\eta$, $l$. The output is the estimated feasible point, $x$.

\noindent\textbf{Algorithm $ E_k: $}

 {\tt input:} $ \eta>0 $, $ l\in\N $, $ a,b\in\R^n $, $ f,g_1,\dots,g_s\in\R $[$ x $].

 {\tt output:} $ x\in\R^n $ (as an approximate solution to GPO Problem~\eqref{eq:GPOAsli}).

 {\tt Initialize:}

\emph{ $\quad a(0) \leftarrow a; \;\;b(0)\leftarrow b ; \;\; m\leftarrow0;\;\;
\lambda(0)\leftarrow B_k(a(0),b(0));$}

{\tt While} ($ m<l $)$ \;:\{ $
\emph{\begin{align}
\label{const: choiceOfBranch}
j^*\leftarrow&\underset{j\in\{0,\dots,m\}}{\arg\min} \,\lambda(j)\nonumber;\hspace{ 2.1 in}\\
i^*\leftarrow&\underset{j\in\{0,\dots,m\}}{\arg\min}   \quad  \prod_{i=1}^{n}(b(j)_i-a(j)_i) \nonumber\\
&\text{subject to }\quad  \lambda(j)\le \lambda(j^*)+\frac{m\,\eta}{1+l}; \\
a^*\leftarrow&a(i^*);\quad\quad b^*\leftarrow b(i^*); \nonumber\\
r^*\leftarrow&\underset{j\in\{1,\dots,n\}}{\arg \max} (b^*_j-a^*_j); \quad\tilde{a} \leftarrow a^*; \quad\hat{b}\leftarrow b^*\nonumber;
\end{align}
\begin{align*}
&\text{{\tt For }}\text{{\tt $ r $ from 1 to $ n :\;\{$}}\\
\tilde{b}_{r}&\leftarrow\begin{cases} \frac{b^*_r+a^*_r}{2} & \text{if }r = r^*\\ b^*_r & \text{otherwise}\end{cases} ;\quad \hat{a}_{r}\leftarrow\begin{cases} \frac{b^*_r+a^*_r}{2} &\text{if } r = r^*\\ a^*_r & \text{otherwise,}\end{cases}; \}
\end{align*}}
\begin{align*}
& \tilde{\lambda}\leftarrow B_k(\tilde{a},\tilde{b}); &&\hat{\lambda}\leftarrow  B_k(\hat{a},\hat{b}); &&& m\leftarrow m+1;\\
& a(i^*)\leftarrow\tilde{a};  && b(i^*)\leftarrow\tilde{b};  &&&\lambda(i^*)\leftarrow\tilde{\lambda};\\
& a(m)\leftarrow\hat{a}; && b(m)\leftarrow\hat{b}; &&&\lambda(m)\leftarrow\hat{\lambda};\}
\end{align*}

{\tt Return}  $ x:=\frac{a(l)+b(l)}{2}$;


In the following section we will discuss the complexity and accuracy of the sequence of Algorithms $ E_k $.

\section{Convergence and Complexity of $E_k$}
\label{section: Complexity}
In this section, we first show that for any $ k\in\N $, the greatest lower bounds obtained by the subroutine $B_k$ increase at each iteration of the Branch and Bound loop. Next we show that for any desired accuracy, there exists a sufficiently large $ k $, such that Algorithm $ E_k $ returns a proposed solution with that accuracy.

In the following lemma, we use Lemma~\ref{lem:abcd} from the Appendix to show that for any $ a_1,\;a_2,\;b_1,\;b_2\in\R^n $ such that $ a_1\le a_2 < b_2\le b_1\in\R^n $, the feasible set of the SOS problem solved in Subroutine $ B_k(a_1,b_1) $ is contained in that of Subroutine $ B_k(a_2,b_2) $.

\begin{mylem}
	\label{lem:C1C2}
	For any $ k\in\N$ and $ a\le b\in\R^n $, let $ M^{(k)}_{a\,b} $ be the modified degree-k bounded quadratic module associated to polynomials $ g_1,\dots,g_s $, as defined in~\eqref{eq:ModifiedQuadraticModuleK}. If $ \gamma \le \alpha<\beta\le \delta\in\R^n $, then  $M^{(k)}_{\gamma\,\delta}\subset M^{(k)}_{\alpha\,\beta}$, for all $ k\ge 2 $.

\end{mylem}
\begin{proof}
	
	For any $ j=1,\dots,n, $ let $w_{j,1}(x):=(\beta_j-x_j)(x_j-\alpha_j),$ and $w_{j,2}(x):=(\delta_j-x_j)(x_j-\gamma_j)$.  Since $ \gamma_j\le\alpha_j<\beta_j\le\delta_j,$  then it is followed from Lemma~\ref{lem:abcd} in the Appendix that  there exist $p_j,q_j,r_j\in\mathbb{R}$ such that
	\[
	w_{j,2}(x) = p_j^2\; w_{j,1}(x) \;+\; q_j^2 \; (x_j+r_j)^2.
	\]
	Now, if $h\in M^{(k)}_{\gamma\,\delta} $, we will show that $ h\in M^{(k)}_{\alpha\,\beta} $. By definition, there exist $\sigma_i,\omega_{j,2}\in \Sigma_S$ such that $h=\sum_{i=0}^{s}\sigma_i g_i+\sum_{j=1}^{n}\omega_j w_{j,2}$, where $ g_0(x)=1 $. Hence,  we can plug in the expression for $w_{j,2}$ to get
	\begin{align*}
	&h = \sum_{i=0}^{s}\sigma_i\cdot g_i+\sum_{j=1}^{n}\omega_j\cdot (p_j^2\cdot w_{j,1} + q_j^2\cdot (x_j+r_j)^2)\\
	=& \underbrace{(\sigma_0+\sum_{j=1}^{n}q_j^2\cdot\omega_j\cdot (x_j+r_j)^2)}_{\sigma_{0\text{ new}}}+\sum_{i=1}^{s}\sigma_i\cdot g_i+\sum_{j=1}^{n}\underbrace{p_j^2\,\omega_j}_{\omega_{j\text{ new}}}\cdot w_{j,1}.
	\end{align*}
	Clearly $\sigma_{0\text{ new}}, \omega_{j\text{ new}} \in \Sigma_{S}$. Furthermore, since $ k\ge2 $,  $\deg(\sigma_{0\text{ new}})\le k$, and $\deg(\omega_{j\text{ new}}\cdot w_{j,1})\le k$ which implies that $h\in M^{(k)}_{\alpha\,\beta}$.
\end{proof}

Now suppose $\{g_i\}$ all have degree $ d $ or less. Then for any $ k\ge d+2 $ and for any hyper-rectangles $C(c,d)\subset C(a,b)$, if $ \lambda_{(a,b)} $ and $ \lambda_{(c,d)} $ are the solutions obtained by Subroutines $B_k(a,b) $ and $ B_k(c,d)$ applied to GPO Problem~\eqref{eq:GPOAsli}, then  Lemma~\ref{lem:C1C2} shows that $ \lambda_{(a,b)}\le \lambda_{(c,d)} $. Now, for a fixed $ k\in\N $,  let $ \eta $ and $ l $ be the design parameters of Algorithm $ E_k $ applied  to GPO Problem~\eqref{eq:GPO}. For $ m=0,\dots,l $, let $ (\lambda^*)_m:=\lambda(j^*) $,  where $ j^* $ is as we defined in iteration $ m $ of the loop in Algorithm $ E_k $. Using Lemma~\ref{lem:C1C2}, it is straightforward to show that $ (\lambda^*)_m\le(\lambda^*)_{m+1} $ for $ m\le l-1 $.

In the next theorem, we will show that for any  given $ \epsilon>0 $, there exist $ k\in\N $ such that Algorithm $ E_k $ applied to GPO Problem~\eqref{eq:GPOAsli} will  provide a point $ x\in\R^n $ satisfying~\eqref{prop:output}.

\begin{mythm}
	\label{theo: complexity}
	Suppose GPO Problem~\eqref{eq:GPOAsli} has a nonempty and compact feasible set $ S $. Choose $ a,b\in\R^n $ such that $S\subset C(a,b)$. For any desired accuracy, $ 0<\epsilon<1 $, let $l > n \log_2(\frac{L\sqrt{n}}{\eta})$ and $\eta <\epsilon$ where $L=\max_i b_i-a_i$. Then there exists a $k \in \N$ such that if $x=E_k(\eta,l,a,b,f,g_i)$, then there exists a feasible point $y\in S $ such that $f(y)-f^* \leq \epsilon$ and $\Vert y-x\Vert<\epsilon$, where $ f^* $ is the OBV of GPO Problem~\eqref{eq:GPOAsli}.
\end{mythm}

\begin{proof}
Define $ \mathcal{P} $ to be the set of all possible hyper-rectangles generated by the branching loop of Algorithm $ E_k$ (for any $k$) with number of branches bounded by $l$. The vertices of all elements of $ \mathcal{P}$ clearly lie on a grid with spacings $ \frac{|a_i,b_i|}{2^l} $. Therefore, the cardinality $ |\mathcal{P}| $ is finite and bounded as a function of $l$, $a$ and $b$. It has be shown that for any $ C_\alpha:=C(e,f)\in\mathcal{P} $, there exists a $ k_\alpha\in\N $ such that for any $ k'\ge k_\alpha $, the solution of Subroutine $B_{k'}( e,f)$ is accurate with the error tolerance $\frac{\eta}{1+l} $. Now define $ k:=\max\{ k_\alpha\;| \; C_\alpha\in\mathcal{P}\} $.
	
Will now show that Algorithm $ E_{k} $  returns a point $x$ with the desired accuracy. First, we show that Algorithm $ E_k $ generates exactly $ l $ nested hyper-rectangles. The proof is by induction on $ m $.

	For $ m=0,\dots,l-1 $, let $ (a)_m:=a(i^*) $, $ (b)_m:=b(i^*) $, $ (C)_m:=C((a)_m,(b)_m) $, $ (\lambda)_m := B_k((a)_m,(b)_m)$, $ (\tilde{a})_m:=\tilde{a}$, $(\tilde{b})_m:=\tilde{b} $,  $ (\hat{a})_m:=\hat{a}$, $(\hat{b})_m:=\hat{b} $, $ (\tilde{C})_m:=C((\tilde{a})_m,(\tilde{b})_m)$, $(\hat{C})_m:=C((\hat{a})_m,(\hat{b})_m) $ and $(\lambda^*)_m := \lambda(j^*) $ where $ i^* $, $ j^* $, $ \tilde{a} $, $ \tilde{b} $, $ \hat{a} $ and $ \hat{b} $ are defined as in iteration $ m $ of Algorithm $ E_k $.

	We use induction on $ m $ to show that for all $ m\le l $ :
	\[(C)_{m}\subset(C)_{m-1}.\]
	The base case $ m=0 $ is trivial. For the inductive step, first note that $ (\lambda^*)_m\le f^*$ for all $ m\le l $ and $ (\lambda^*)_1\le\dots\le(\lambda^*)_{l} $. The latter is obtained from Lemma~\ref{lem:C1C2} and the former is because at each iteration, $S\subset\bigcup_{i=0}^{m} C(a(i),b(i))$.
	Constraint~\eqref{const: choiceOfBranch} at iteration $ m $, implies that
	\begin{align}
	\label{eq:namosavie4}
			& B_k((a)_m,(b)_m)\le (\lambda^*)_m+\frac{m\,\eta}{l+1}.
	\end{align}
	
	Now we will show that again, Constraint~\eqref{const: choiceOfBranch} at iteration $ m+1 $ is satisfied at least by one of $ (\tilde{C})_m $ and $ (\hat{C})_m $. Suppose this is not true. Then we can write
	\begin{align}
	\label{eq:inequalitiesProof}
	(\lambda^*)_{m+1}<&  B_k((\tilde{a})_m,(\tilde{b})_m)-\frac{(m+1)\,\eta}{1+l},\\
	(\lambda^*)_{m+1}<&  B_k((\hat{a})_m,(\hat{b})_m)-\frac{(m+1)\,\eta}{1+l}\nonumber .
	\end{align}
	Now, since $ (\lambda^*)_{m+1}\ge(\lambda^*)_{m} $, Eq.~\eqref{eq:inequalitiesProof} implies
	 \begin{align}
	 \label{eq:inequalities 3}
	 (\lambda^*)_{m}<&  B_k((\tilde{a})_m,(\tilde{b})_m)-\frac{(m+1)\,\eta}{1+l},\\
	 (\lambda^*)_{m}<&  B_k((\hat{a})_m,(\hat{b})_m)-\frac{(m+1)\,\eta}{1+l}\nonumber .
	 \end{align}
	 Using Eq.~\eqref{eq:inequalities 3} and Eq.~\eqref{eq:namosavie4} one can write
	 \begin{align}
	 \label{eq:akharin}
	 	B_k((a)_m,(b)_m) <& B_k((\tilde{b})_m,(\tilde{b})_m)-\eta/l,\\
	 	B_k((a)_m,(b)_m) <& B_k((\hat{a})_m,(\hat{b})_m)-\eta/l\nonumber.
	 \end{align}
	 This contradicts the fact that all $ B_k((\tilde{a})_m,(\tilde{b})_m) $, $ B_k((\hat{a})_m,(\hat{b})_m) $ and $B_k((a)_m,(b)_m)$ have  accuracy higher than $ \eta/(1+l) $. Therefore, it is clear that  both $ (\tilde{C})_m $ and $ (\hat{C})_m$ can be possible choices to be bisected at  iteration $ m+1 $. This fact, together with the induction hypothesis which certifies that $ (C)_m $ possesses the smallest volume between all the hyper-rectangles obtained up to that iteration, guaranteeing that either $ (\tilde{C})_m $ or $ (\hat{C})_m $, will be branched at the next iteration. Therefore, the algorithm will generate $ l $ nested hyper-rectangles.

	 Now,  $ (\lambda^*)_0\in[f^*-\eta/l,f^*] $ implies that
	 \begin{align}
	 \label{eq:namosavi}
	 &(\lambda^*)_m\in[f^*-\eta/l,f^*],\;\text{ for all }  m=1,\dots,l.
	 \end{align}
	 Eq.~\eqref{eq:namosavie4} and Eq.\eqref{eq:namosavi} together with the fact that $ (\lambda)_m\ge(\lambda^*)_m $ imply
	 \begin{align}
	 \label{eq:namosavi2}
	 &(\lambda)_m\in\left[f^*-\eta/l,f^*+\frac{m\,\eta}{l+1}\right],\;\forall\;m=1,\dots,l.
	 \end{align}
	 Finally, as a special case $ m=l $, one can write:
	 \[ f^*-\frac{\eta}{1+l} \le B_k((a)_{l},(b)_{l}) \le f^*+\frac{l\,\eta}{l+1}. \]
	 Now, note that the $ \eta/(l+1) $-accuracy of $ B_k((a)_{l},(b)_{l})$ implies that  $ (C)_{l} $ is feasible. It also can be implied that $ (C)_{l}\cap S $ contains $ y $ such that $f(y) \ge B_k((a)_{l},(b)_{l} \ge f(y)-\eta/(l+1) $. Therefore, $|f(y)-f^*|\le \eta\le\epsilon.$
	
	 Finally, if $ x $ is the point return by Algorithm $ E_k $, then based on the definition of $ l $, it is implied that after the last iteration $ m=l-1 $, the largest diagonal of the branched hyper-rectangle is less than $ \eta\le\epsilon $, hence $ \Vert y-x\Vert_2\le\epsilon $, as desired.
\end{proof}

 Theorem~\ref{theo: complexity} ensures that for any accuracy $ \epsilon>0 $ there exists a $ k\in\N $ such that Algorithm $ E_k $ returns $\epsilon$-approximate solutions to the GPO problem with a logarithmic bound on the number of branching loops. The following corollary shows that these $\epsilon$-approximate solutions can themselves approximately satisfy the constraints of the original GPO as follows.
 \begin{mycor}
 \label{corollaryLip}
 Let GPO Problem ~\eqref{eq:GPOAsli} have nonempty and compact feasible set that is contained in $ C(a,b)\, $ for some $ a,b\in\R^n $. For any given $ \delta> 0 $, there exists $ \epsilon>0 $ such that if $\epsilon$ and  $ x= E(\eta,l,a,b,f,g_i) $ satisfy the conditions in Theorem~\ref{theo: complexity}, then
\begin{align}\label{eq:Last}
|f(x)-f^*|\le\delta \text{ and }\quad g_i(x)\ge-\delta,\;\forall i=1,\dots,s.
\end{align}

 \end{mycor}
 \begin{proof}
 Let $ L$ be such that any polynomial $ h \in\{ f,g_1,\dots,g_s\} $ satisfies
 $|h(c)-h(d)|\le L|c-d|_{2},\;\forall c,d\in C(a,b)$. ( Existence of $ L $ follows from the Lipschitz continuity of polynomials on compact sets.) Choose $ \epsilon $ such that $ \epsilon\le \delta/L $. Let $ \epsilon $ and $ x $ satisfy the conditions in Thoerem~\ref{theo: complexity}. It is straightforward to show that  $ x $ satisfies Eq.~\eqref{eq:Last}.
 \end{proof}

Unfortunately, of course, Theorem~\ref{theo: complexity} does not provide a bound on the size of $ k $ (although the proof implies an exponential bound).\\
%
%
%
%

\section{Numerical Results}
\label{section: NE}
Consider the following GPO problem.
\begin{align*}
\quad \underset{x\in\mathbb{R}^6}{min.} & \quad  f(x)=7x_1 x_5 ^3 + 6x_1 x_5 ^2 x_6 + 9x_2 x_4 ^3 + 4 x_2 x_4 x_5 +\\ & \quad\quad\quad\quad 3x_2 x_5 x_6 + x_3 x_4 x_5\\
\text{subject to } & \;\;\text{ $ g_1(x) =100-(x_1 ^2+x_2 ^2+x_3 ^2+x_4 ^2+x_5 ^2+x_6 ^2) \ge 0$}\\
& \;\; \;\,\text{$ g_2(x)=x_1^ 3 + x_2 ^2 x_4 + x_3x_5^2 \ge 0$}\\
& \;\; \;\,\text{$ g_3(x)= x_2^2x_1+x_5^3+x_4x_1x_2 \ge 0$}\\
& \;\; \;\,\text{$ h_1(x)= x_1+x_2^2-x_3^2+x_4x_5 = 0$}\\
& \;\; \;\,\text{$ h_2(x)=x_5x_1-x_4^2 = 0$}
\end{align*}
In this example we have 6 variables, an objective function of degree 4 and several equality and inequality constraints of degree 4 or less. The ideal generated by equality constraints is not zero dimensional, hence the Moment approach to extracting solutions fails. We applied Algorithm $ E_5 $ to this problem with parameters $ \eta=0.005 $ and $ l=200 $, using Sedumi to solve the SDPs associated with the SOS and Moment problems.  As seen in Figure~\ref{fig:figure}, the branch and bound algorithm converges relatively quickly to a certain level of error and then saturates. Iterations past this point do not significantly improve accuracy of the feasible point. As predicted, this saturation and residual error (blue shaded region) is due the use of a fixed degree bound $k=5$. As $k$ is decreased, the residual error increases and as $k$ is increased the residual error decreases. For this problem the final iteration returns the point $\hat{x}= [5.1416,\;3.9307,\;   0.7568,\;   -4.6777,\;    4.2676,\;   -4.1504 ] $ for which all inequalities are feasible and the equality constraints $h_1$ and $h_2$ have errors of $0.0563$ and $0.0610$, respectively. The objective value is $f(\hat{x})=-3693.3$.

\begin{figure}[ht]
	\centering
	\subfigure[Sum of errors in inequality constraints
	at center, $\sum_{g_i(\mathbf{x}_m)\le 0}| g_i(\mathbf{x}_m)|$  vs. number of iteration]{
	\hspace{-0.11in}	\includegraphics[scale=0.21]{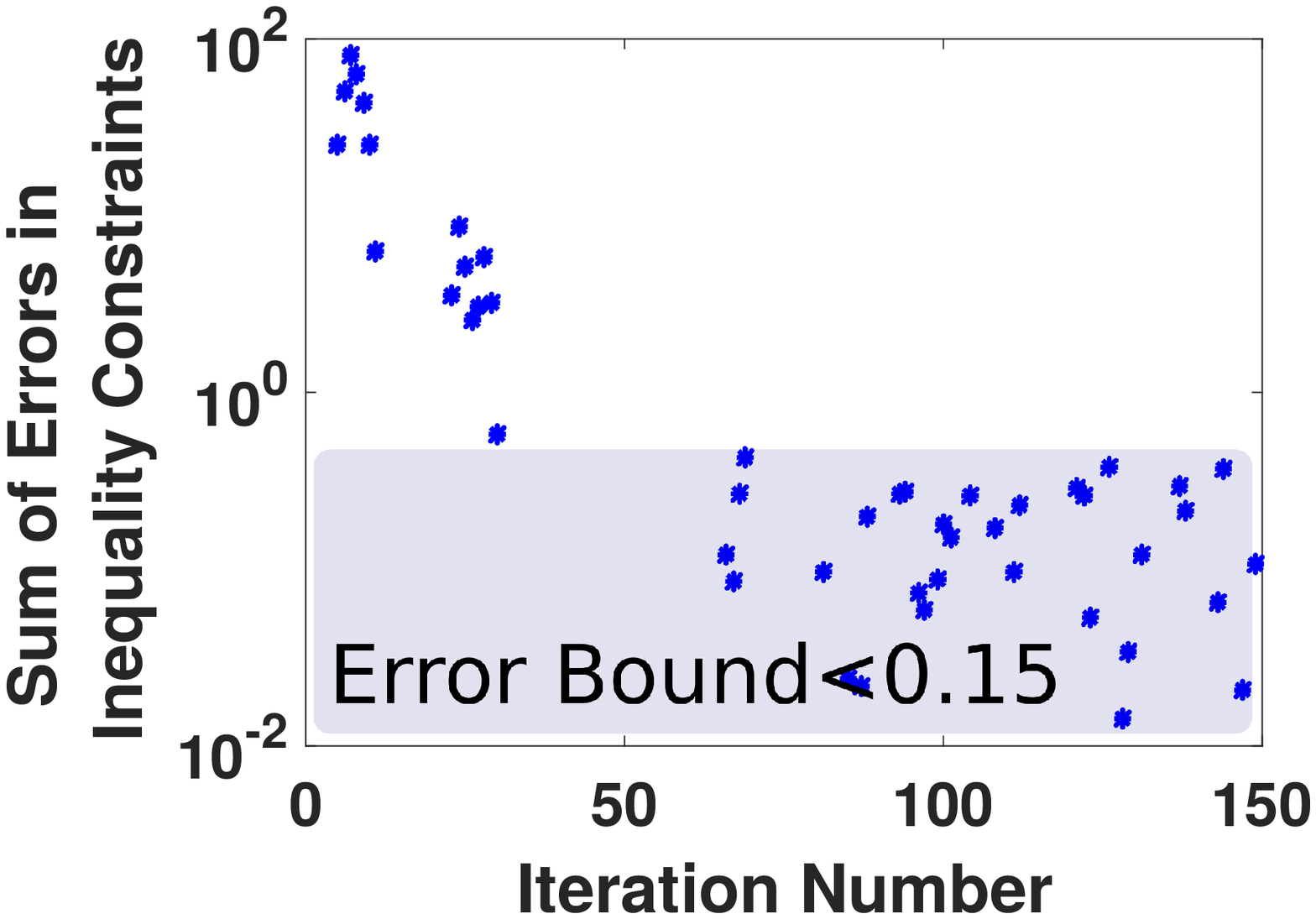}
		\label{fig:subfigure1}}
	\hspace{+0.00in}
	\subfigure[ Sum of errors in equality constraints
	at center, $\sum_{j}| h_j(\mathbf{x}_m)|$  vs. number of iteration]{%
		\includegraphics[scale=0.21]{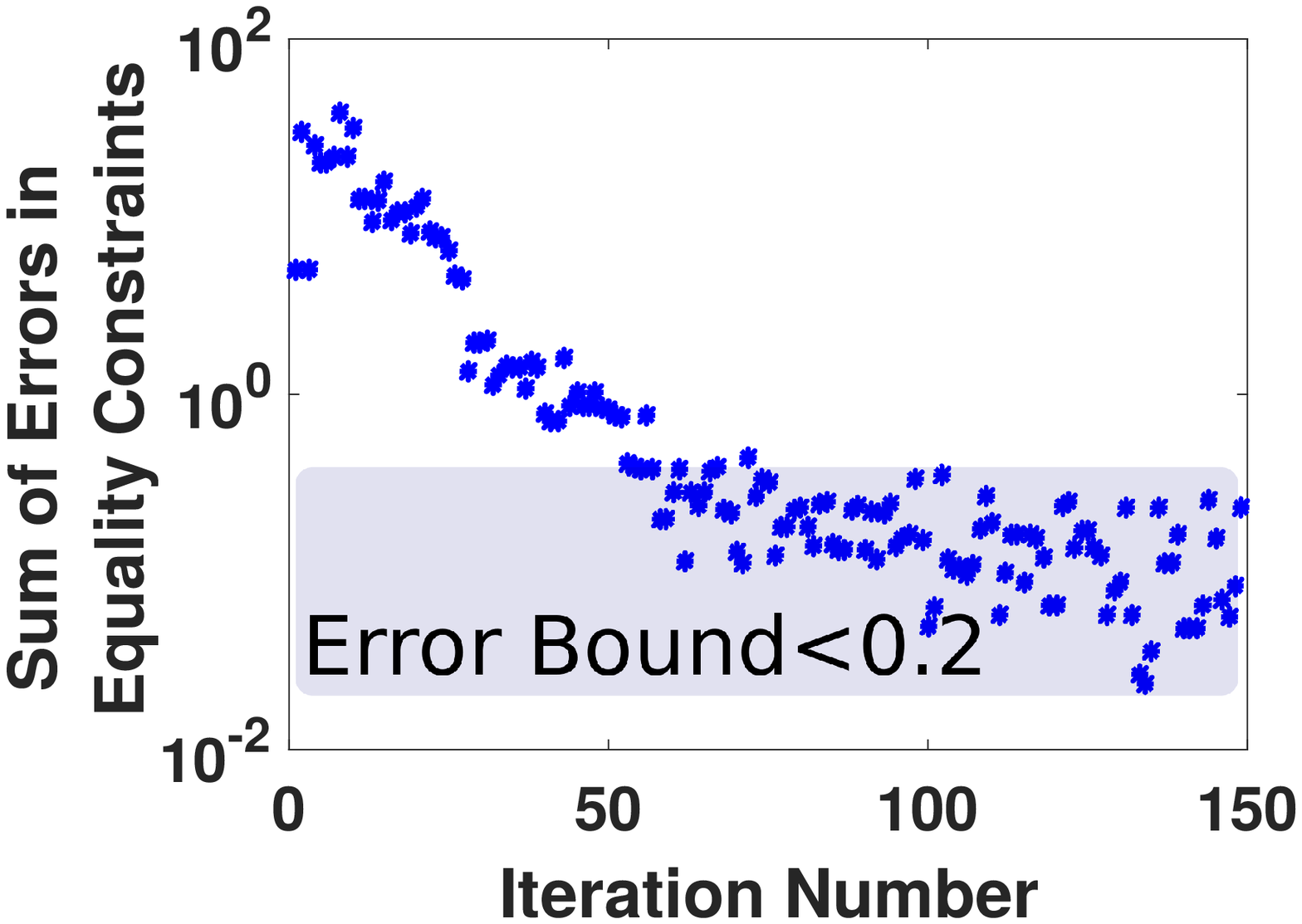}
		\label{fig:subfigure2}}
	\subfigure[Gap between lower bound and objective
	value at center point, $|f(\mathbf{x}_m)-l|$ vs. number of iteration]{%
	\hspace{-0.15in}	\includegraphics[scale=0.21]{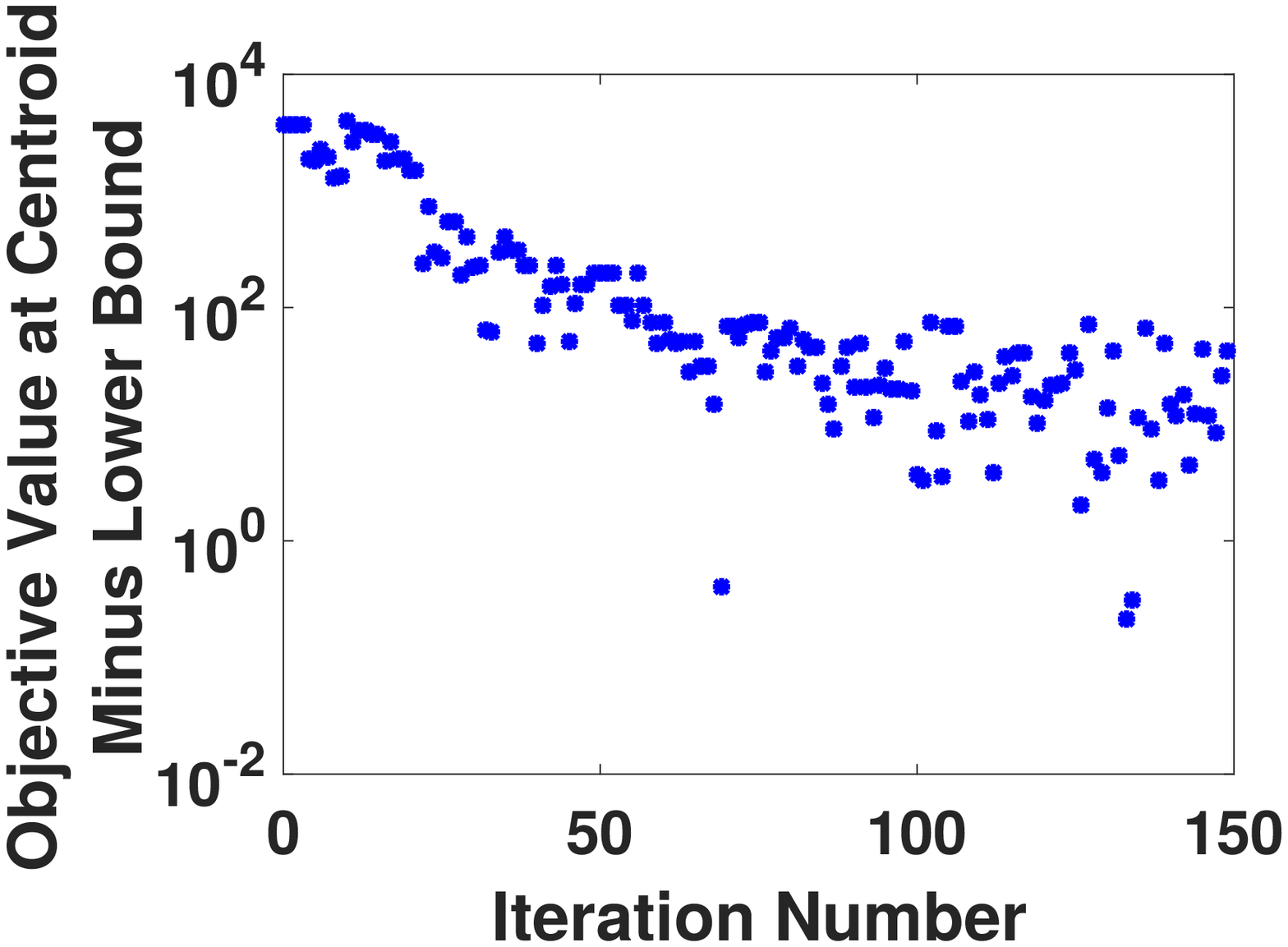}
		\label{fig:subfigure3}}
	\hspace{+0.04in}
		\subfigure[Lower bound vs. number of iteration]{%
			\hspace{-0.11in}\includegraphics[scale=0.21]{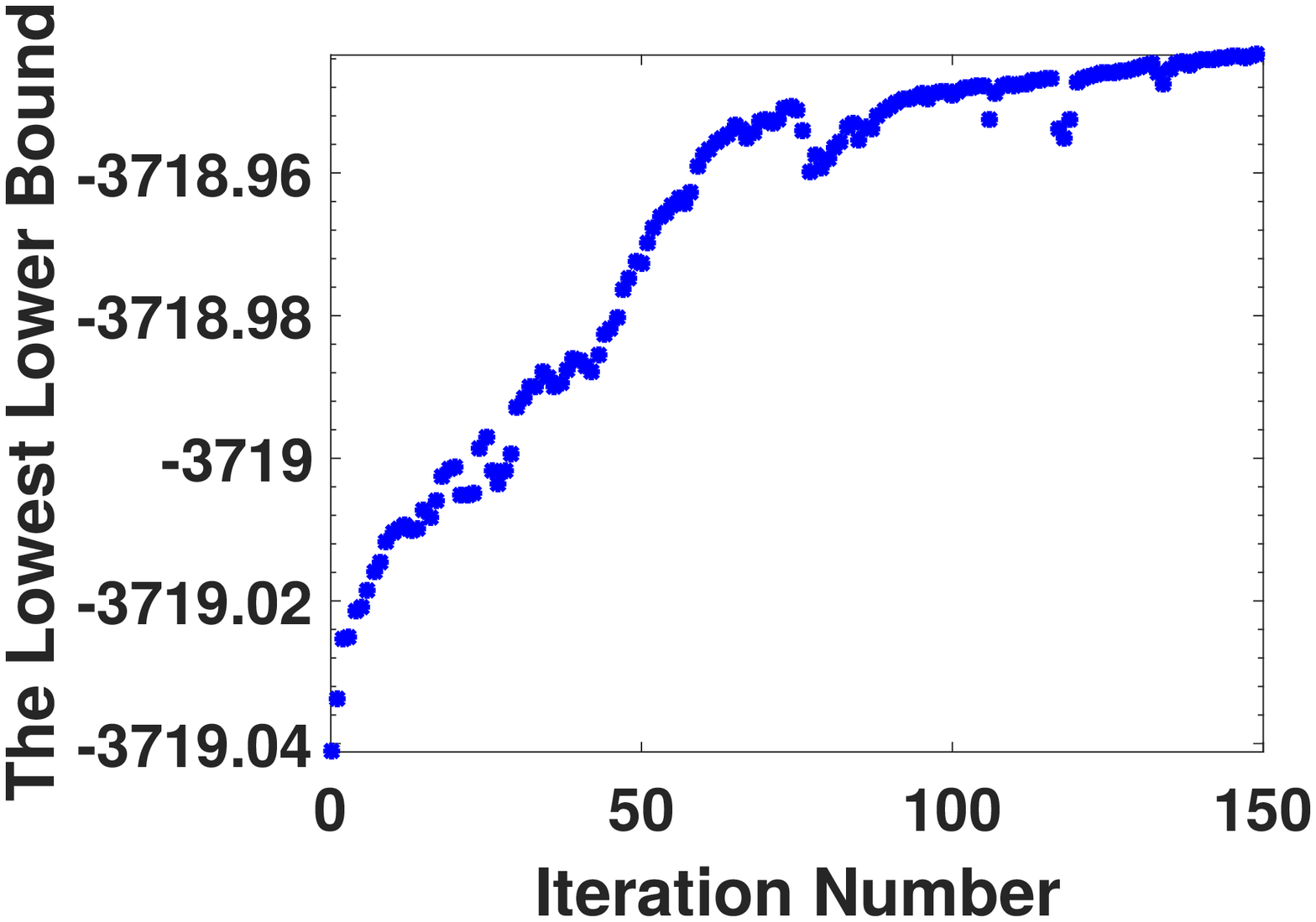}
			\label{fig:subfigure6}}
	\subfigure[Lower bound vs. number of iteration]{%
			\hspace{-0.11in}\includegraphics[scale=0.21]{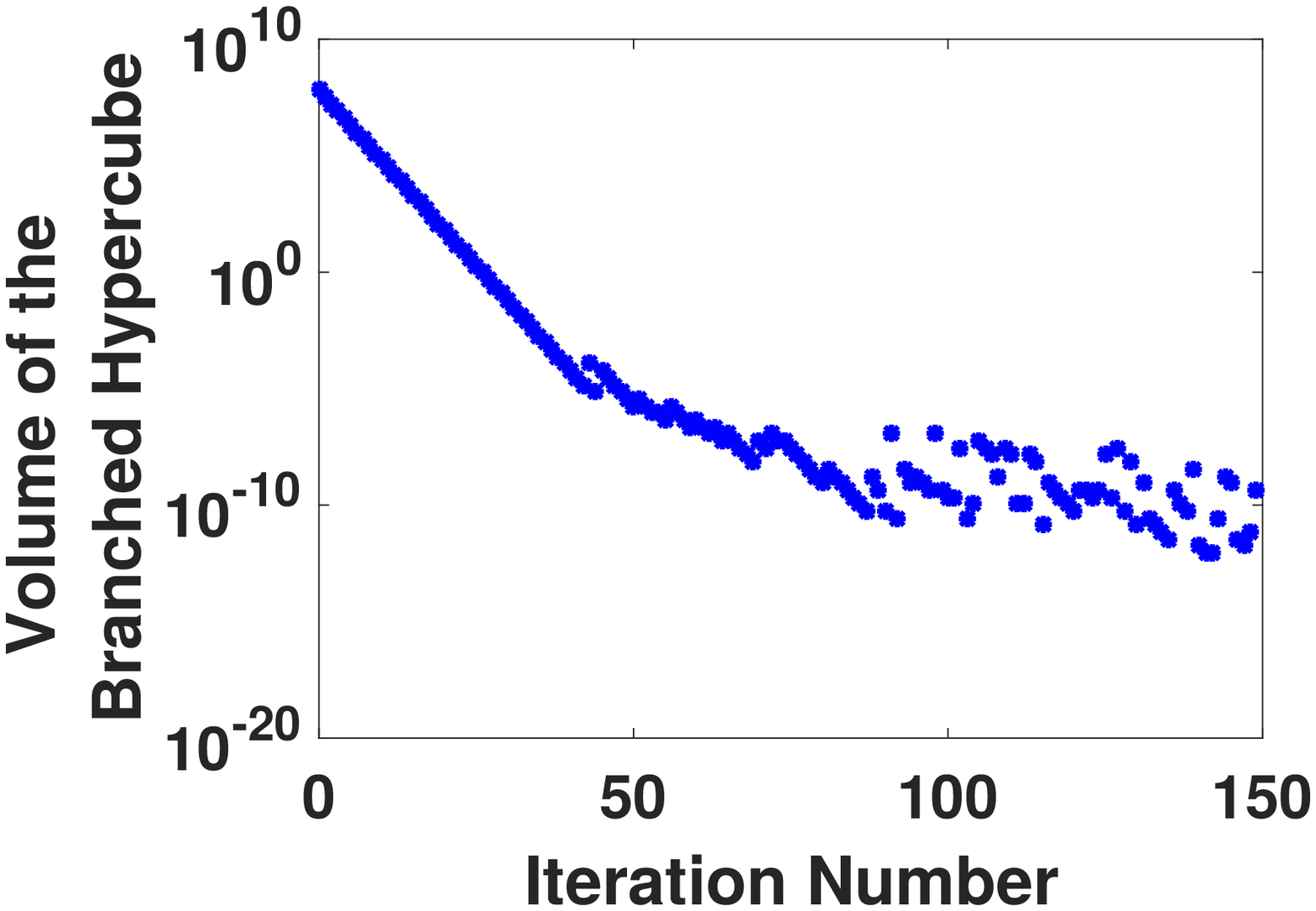}
			\label{fig:subfigure7}}
	\subfigure[Longest edge of the branched hypercube vs. number of iteration]{%
		\includegraphics[scale=0.21]{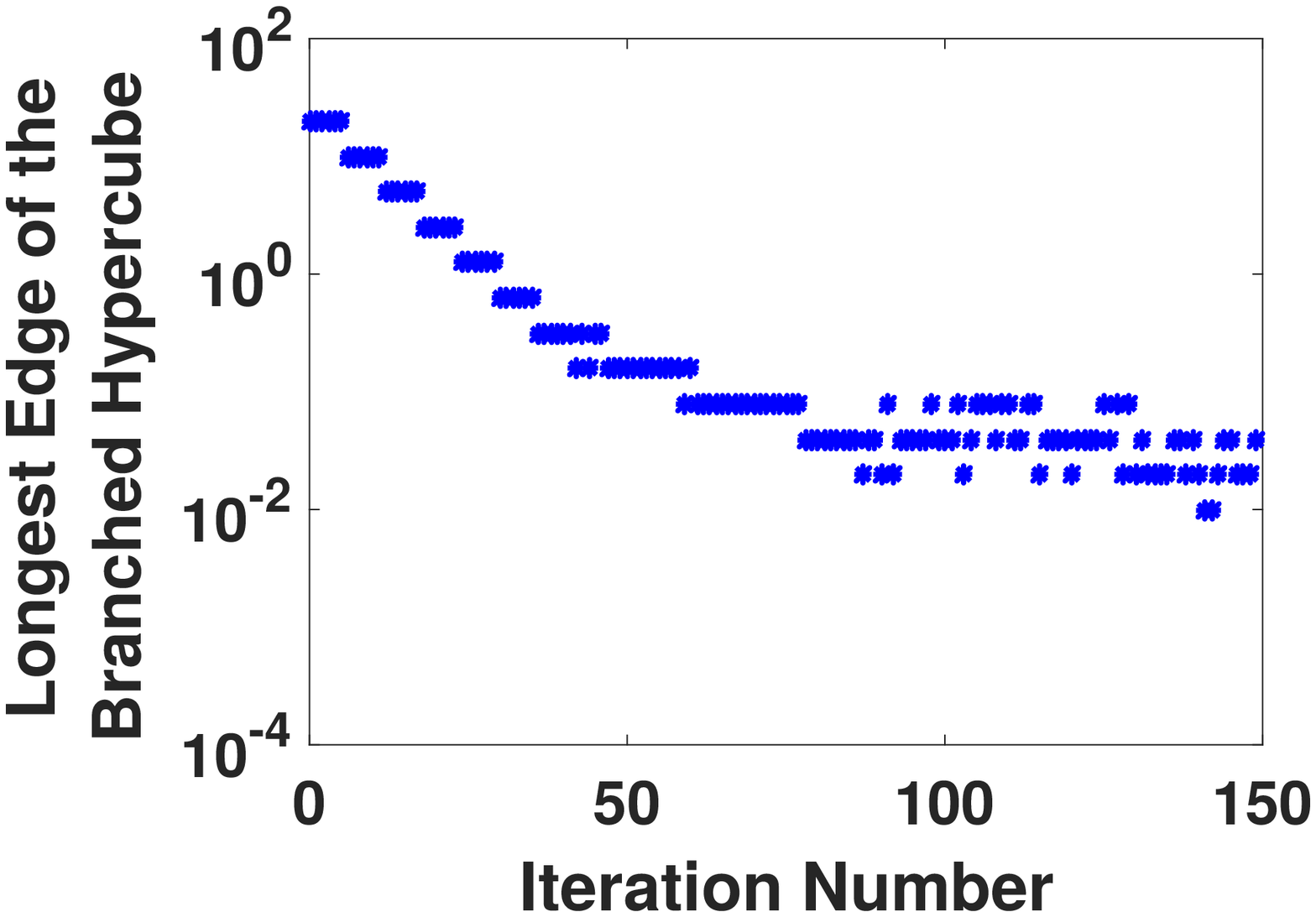}
		\label{fig:subfigure4}}
	\quad
	\subfigure[Objective function value at center point vs. number of iteration]{%
		\hspace{-0.11in}\includegraphics[scale=0.21]{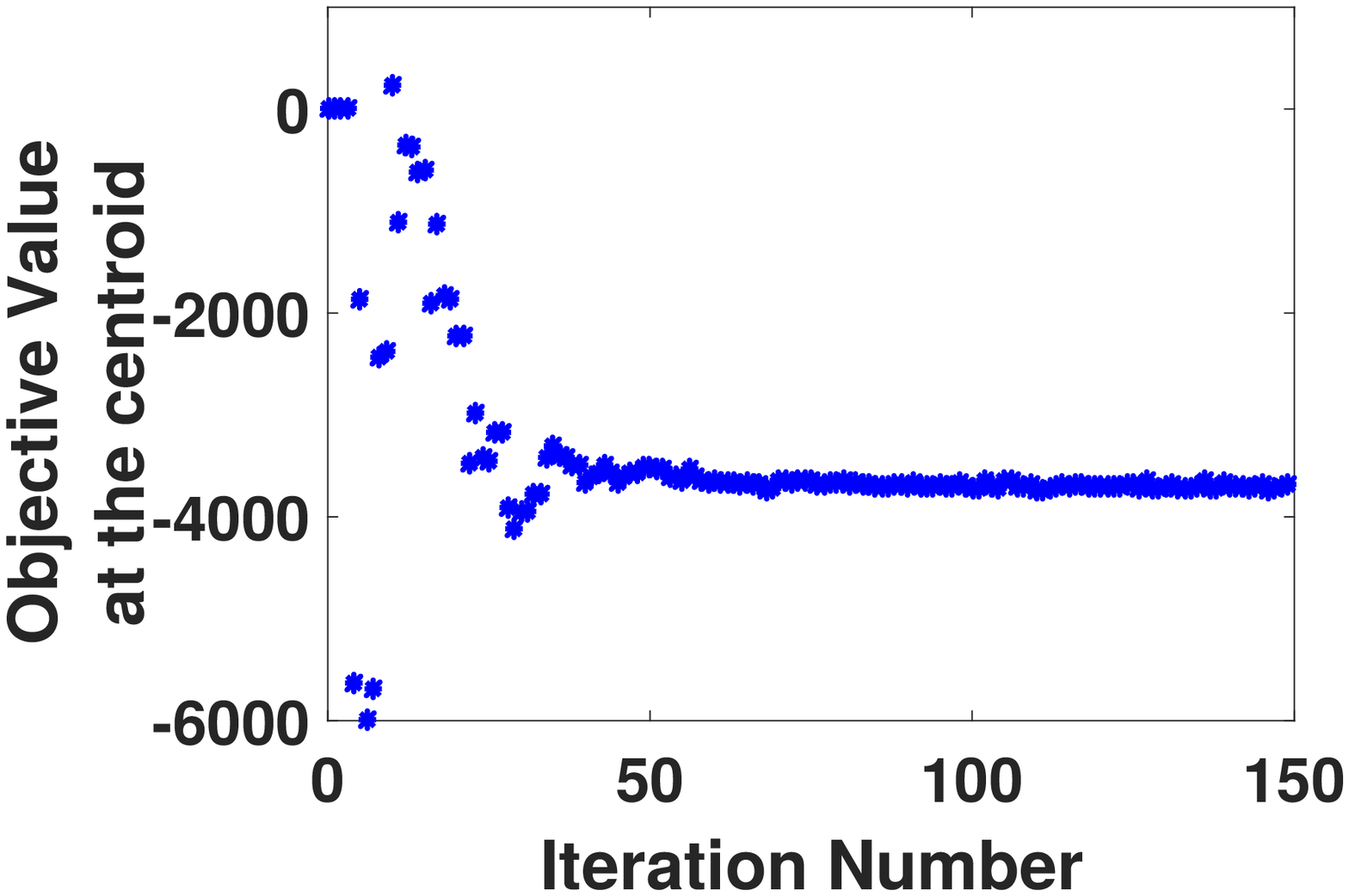}
		\label{fig:subfigure5}}
		\quad

	\caption{Numerical results of the example}
	\label{fig:figure}
\end{figure}

\section{Conclusion}
\label{section:Conclusion}
We have proposed a sequence of Algorithms  $ E_k,\;k\in\N$ to extract solutions to the GPO problem based on a combination of Branch and Bound and SOS/Moment relaxations. The computational-complexity of Algorithm $ E_k $ is polynomial in  $ k $, polynomial in the number of constraints and linear in the number of branches $ l $. Additionally, for any scalar $ \epsilon>0$, there exist $ k\in\N $ such that Algorithms $ E_k$, in $ O(\log(1/\epsilon)) $ number of iterations,  returns a point that is within the $ \epsilon $-distance of  a feasible and $\epsilon$-suboptimal point.
For a fixed degree of semidefinite relaxations, our numerical case study demonstrates convergence to a level of residual error which can then be decreased by increasing the degree.
In ongoing work, we seek to bound this residual error as a function of degree using available bounds on the error of SOS/Moment relaxations.

\appendix
The following lemma gives an algebraic property of the polynomials of the form $ w(x) =(x-a_i)(b_i-x) $ which are used to define the augmented feasible set $S_{ab}$.
\begin{mylem}
	\label{lem:abcd}
	Let $a\le c<d\le b \in \mathbb{R}$, $g := (x-a)(b-x)$ and $h := (x-c)(d-x)$. Then, there exist $\alpha, \beta$ and $\gamma\in \mathbb{R}$, such that
	\[
	g(x) = \alpha h(x) + \beta (x+\gamma)^2\,,\quad \alpha ,\beta \ge 0\vspace{0.06in}
	\]
\end{mylem}

\begin{proof}
	Without loss of generality, one can assume that $a = 0$ (consider the change of variable $z := x-a$). Now let $p^2:=c$, $q^2:= d-c$, and $r^2:=b-d$.
	First, we consider the case where $p^2,r^2\neq 0$. This leads to two sub-cases:
	
\noindent \textbf{Case 1 :} $r^2\neq p^2$. Let
	\begin{footnotesize}
		\[
		\gamma=  \frac{p^4+p^2q^2-\sqrt{p^2r^2(p^2+q^2)(q^2+r^2)}}{r^2-p^2}\,,\,
		\beta = \frac{p^4+p^2q^2}{\gamma^2-p^4-p^2q^2}\,,\,
		\]
	\end{footnotesize}
	and $\alpha = \beta+1$. Verifying the equality $g(x) = \alpha h(x) + \beta (x+\gamma)^2$ is straightforward. To show that $\beta,\alpha \ge 0$, we use the following.
	\begin{footnotesize}
		\[\beta\ge 0 \iff  \gamma^2> p^4+p^2q^2 \]
		\[\iff\Big(p^4+p^2q^2-\sqrt{p^2r^2(p^2+q^2)(q^2+r^2)}\Big)^2>(p^4+p^2q^2)(r^2-p^2)^2\]
		\[
		\iff\underbrace{(p^4+p^2q^2)^2+p^2r^2(p^2+q^2)(q^2+r^2)-(p^4+p^2q^2)(r^2-p^2)^2}_{L}>
		\]
		\[
		\underbrace{2\big(p^4+p^2q^2\big)\sqrt{p^2r^2(p^2+q^2)(q^2+r^2)}}_{U}\iff\begin{cases}
		L>0\\
		L^2>U^2
		\end{cases}
		\]
	\end{footnotesize}
	After simplification we have:
	
	\begin{footnotesize}
		\[L^2-U^2 = p^4q^4(p^2+q^2)^2(p^2-r^2)^2> 0,\quad \text{and }\]
		\[L = p^2\, \left(p^2 + q^2\right)\, \left(p^2\, q^2 + 2\, p^2\, r^2 + q^2\, r^2\right)>0\]
	\end{footnotesize}
		which completes the proof for Case 1.\vspace{0.05in}\\
\noindent \textbf{Case 2 :} $r^2 = p^2$. In this case, let
	\begin{footnotesize}
		\[
		\gamma = -\frac{2p^2+q^2}{2}\,,\,
		\beta = \frac{4p^2(p^2+q^2)}{q^4}\,,\,
		\alpha = \beta+1
		\]
	\end{footnotesize}
Equality and positivity for this case can then be easily verified. Now, suppose $r^2=p^2=0$. In this case, simply set $\beta=0,\,\alpha=1$. If $p^2=0,r^2\neq 0$, set $\beta=\frac{b}{d}-1,\,\alpha=\frac{b}{d},\,\gamma=0$. The case $p^2\neq 0$, $r^2= 0$ is similar to $p^2=0$, $r^2\neq 0$,through the change of variable $z=b-x$.
\end{proof}

\bibliographystyle{ieeetr}
\bibliography{ReferencesIFAC2017}
\end{document}